\documentclass{rhumj} 

\usepackage[english]{babel}
\usepackage[utf8x]{inputenc}
\usepackage[T1]{fontenc}
\usepackage[affil-it]{authblk}

\usepackage{amsmath,amssymb}
\usepackage{graphicx}
\usepackage[colorinlistoftodos]{todonotes}
\usepackage[colorlinks=true, allcolors=black]{hyperref}
\usepackage{amsthm, mathtools,amssymb}
\usepackage{graphicx,color}
\usepackage{tabularx}
\usepackage{tikz}
\usepackage{pgfplots}
\usepackage{pst-plot}
\usepackage{pstricks}
\usepgfplotslibrary{fillbetween}
\usetikzlibrary{patterns}
\usetikzlibrary{shadings}
\usetikzlibrary{arrows}
\usetikzlibrary{calc}
\usetikzlibrary{decorations.markings}
\DeclareGraphicsExtensions{.pdf,.eps,.fig,.pdf_tex}
\usepackage{caption,subcaption,float}
\theoremstyle{plain} \numberwithin{equation}{section}
\newtheorem{theorem}{Theorem}[section]

\theoremstyle{definition}
\newtheorem{definition}[theorem]{Definition}

\newtheorem{example}[theorem]{Example}
\title{
Super-Walk Formulae for Even and Odd Laplacians in Finite Graphs
}

\author{
Chengzheng Yu\affiliation{Department of Mathematics, University of Illinois, Urbana, IL 61801, USA, cyu53@illinois.edu}
}

\authornames{
Chengzheng Yu
}

\acknowledgement{
The research in this paper was conducted by the author under the supervision of Ivan Contreras, as part of the Illinois Geometry Lab (IGL) program. The author is thankful for the opportunity to conduct this research, and would also like to thank Sarah Loeb, Zhe Hu, Michael Toriyama, and Boyan Xu for helpful comments.
\\*[\parskip]
\hspace*{1.5em}
}

\abstract{
The number of walks from one vertex to another in a finite graph can be counted by the adjacency matrix. In this paper, we prove two theorems that connect the graph Laplacian with two types of walks in a graph. By defining two types of walks and giving orientation to a finite graph, one can easily count the number of the total signs of each kind of walk from one element to another of a fixed length.
}

\vol{18, No. 1}
\date{Spring 2017}

\begin{document}
\maketitle
%% keep the following line commented out
%\setcounter{page}{19}
%%

%% Your Paper goes here

\section{Introduction}
In graph theory and combinatorics, an interesting problem is that to find the number of different ways a certain operation can be performed on a given graph. For example, given two vertices $v_i$ and $v_j$, a natural question is: how many different walks of length $k$ connect $v_i$ and $v_j$. The answer to this question has interesting application to random walks, spectral graph theory, and surprisingly, discrete versions of quantum mechanics, in which such number of paths appear in the partition function for graph quantum mechanics \cite{Mnev1}. Such physical interpretation can be adapted to the case of super-symmetric quantum mechanics, for which both edges and vertices span the space of quantum states of the system.\\

The particular aim of this paper is to study generalized walks associated to a finite graph that we call of \textbf{super-walks} and \textbf{edge super-walks}. Super-walks differ from the conventional walks between vertices by the fact that staying in a vertex after choosing an edge is allowed, and that such walks come equipped with a natural orientation. Edge super-walks are in some way dual to the super-walks, in the sense that now the walks are considered to be between edges.\\

While powers of the adjacency matrix of a graph can be used to determine the number of conventional walks between vertices, we will show that another matrix associated with a graph, the even Laplacian $\Delta^+$, provides a counting formula for super-walks (Theorem~\ref{SW}). Similar expressions were proved via Feynman-type expansion of the graph Laplacian in the case of regular graphs by Mn\"ev \cite{Mnev1, Mnev2}, and for general finite graphs by del Vecchio \cite{delVeccio}. We provide a combinatorial argument for such formulae, using induction on the length of the super-walks.\\

Edge super-walks appear naturally while studying the super-symmetric version of graph quantum mechanics \cite{Boyan,delVeccio,Mnev1,Mnev2}, for which the walks are considered to be between edges rather than vertices. Theorem~\ref{ESW} provides a counting formula for edge super-walks in terms of the odd Laplacian $\Delta^-$. It turns out that the number of such walks depends on the orientation of the graph, and in Section~\ref{SESW} we will explain how.\\

This paper is organized as follows. In Section~\ref{Walks}, we introduce the different notions of walks and some examples. In Section~\ref{SSW}, super-walk is introduced for non-oriented graphs, and we prove the corresponding counting formula (Theorem~\ref{SW}). In Section \ref{SESW}, we introduce edge super-walks as a super-symmetric version of conventional walks, and we determine their counting formula (Theorem~\ref{ESW}). In the last section, we give an outlook of further directions and generalizations.\\

\section{Walks}\label{Walks}

We start by recalling the usual notion of walk in a finite graph. In this section, Theorem~\ref{walkthm} shows that the number of walks from vertices to vertices of fixed length can be counted by the corresponding adjacency matrix and its powers.\\

A \emph{graph} $\Gamma$ is an ordered pair $(V, E)$ comprising a set $V$ of vertices together with a set $E$ of edges, which are 2-element subsets of $V$ (that is an edge is associated with two vertices, and the association takes the form of the unordered pair of the vertices). If the set $V$ and the set $E$ are finite sets, we say $\Gamma$ is a \emph{finite graph}. A finite graph $\Gamma$ is often indicated by the \emph{adjacency matrix} which is defined by $A_{\Gamma}:=A(i,j) = \left\{\begin{matrix}
1 & \text{if } i\text{ is adjacent to }j\\ 
0 & \text{otherwise}
\end{matrix}\right.$.\\

The \emph{incidence matrix} $I$, the \emph{even graph Laplacian} $\Delta ^{+}_{\Gamma }$, and the \emph{odd graph Laplacian} $\Delta ^{-}_{\Gamma }$ are also used to indicate a graph. Let $\Gamma $ be an oriented graph, with vertices $v_{1},v_{2},\ldots,v_{|V|}$, and edges $e_{1},e_{2},\ldots,e_{|E|}$. The \emph{incidence matrix} $I$ of $\Gamma $ is a $|V|\times |E|$ matrix given by $I(i,j)=\left\{\begin{matrix}
1 & \text{if } e_{j} \text{ ends at } v_{i}\\ 
-1 & \text{if } e_{j} \text{ starts at } v_{i}\\ 
0 & \text{otherwise}
\end{matrix}\right.$. The even graph Laplacian is defined by $\Delta ^{+}_{\Gamma } :=II^{t}$. The odd graph Laplacian is defined by $\Delta ^{-}_{\Gamma } :=I^{t}I$.\\ 

\begin{definition}

Given a graph $\Gamma $ with vertices $v_{1}, v_{2}, \ldots, v_{|V|}$, and edges $e_{1}, e_{2}, \ldots, e_{|E|}$:

(1) A \emph{walk} of length 1 is a walk which starts at $v_{i}$, goes to one of its neighbor $v_{j}$. (*)

(2) A \emph{walk} of length $k$ is a walk which starts at $v_{i}$, repeats (*) $k$ times, and ends at $v_{j}$.

\end{definition}

\begin{example}\label{eg_1}

In Figure~\ref{Figure 1}: (1) $v_{1}\overset{e_{2}}{\rightarrow}v_{4}$ is a walk from $v_{1}$ to $v_{4}$ of length 1; (2) $v_{1}\overset{e_{1}}{\rightarrow}v_{2}\overset{e_{3}}{\rightarrow}v_{4}\overset{e_{4}}{\rightarrow}v_{5}\overset{e_{7}}{\rightarrow}v_{3}\overset{e_{6}}{\rightarrow}v_{6}$ and $v_{1}\overset{e_{2}}{\rightarrow}v_{4}\overset{e_{4}}{\rightarrow}v_{5}\overset{e_{5}}{\rightarrow}v_{6}\overset{e_{6}}{\rightarrow}v_{3}\overset{e_{6}}{\rightarrow}v_{6}$ are two walks from $v_{1}$ to $v_{6}$ of length 5.

\begin{figure}[h]
\begin{center}
	\begin{tikzpicture}
	\tikzset{
	mid arrow/.style={
	decoration={markings,mark=at position 0.5 with {\arrow[scale = 2]{>}}},
	postaction={decorate},
	shorten >=0.4pt}}
	\path (0,0) coordinate (X4); \fill (X4) circle (3pt);
	\path (0,1) coordinate (X1); \fill (X1) circle (3pt);
	\path (1,1) coordinate (X2); \fill (X2) circle (3pt);
	\path (1,0) coordinate (X5); \fill (X5) circle (3pt);
	\path (2,0) coordinate (X6); \fill (X6) circle (3pt);
	\path (2,1) coordinate (X3); \fill (X3) circle (3pt);
	\node[above left] at (X1) {$v_{1}$}; 
	\node[above right] at (X2) {$v_{2}$}; 
	\node[above right] at (X3) {$v_{3}$};
	\node[below left] at (X4) {$v_{4}$};
	\node[below] at (X5) {$v_{5}$};
	\node[below right] at (X6) {$v_{6}$};
	\draw (X1) -- (X2) node[midway,above]{$e_{1}$};
	\draw (X1) -- (X4) node[midway,left]{$e_{2}$};
	\draw (X2) -- (X4) node[midway,right]{$e_{3}$};
	\draw (X4) -- (X5) node[midway,below]{$e_{4}$};
	\draw (X5) -- (X6) node[midway,below]{$e_{5}$};
	\draw (X6) -- (X3) node[midway,right]{$e_{6}$};
	\draw (X3) -- (X5) node[midway,left]{$e_{7}$};
	\end{tikzpicture}
	\end{center}
	\caption{The graph used in Example~\ref{eg_1} and Example~\ref{eg_2}}
\label{Figure 1}\end{figure}
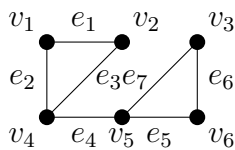

\end{example}

\begin{theorem}\label{walkthm}

Let $A_\Gamma$ be the adjacency matrix of a graph $\Gamma$. $A_\Gamma^k(i,j)$ is the number of \emph{walks} in $\Gamma $ from $v_{i}$ to $v_{j}$.

\end{theorem}

\section{Super-Walks}\label{SSW}

Theorem~\ref{walkthm} gives the combinatorial interpretation of the entries of powers of the adjacency matrix. It is also natural to consider the combinatorial interpretations of entries of graph Laplacians. In this section, we formally define super-walk, and Theorem~\ref{SW} shows that the total signs of super-walks of fixed length can be calculated by the corresponding even graph Laplacian and its powers.\\

\begin{definition}

Given a graph $\Gamma$ with vertices $v_{1}, v_{2}, \ldots, v_{|V|}$, and edges $e_{1}, e_{2}, \ldots, e_{|E|}$, we have:

A \emph{super-walk} of length 1 is a walk which starts at $v_{i_{1}}$, and: 

(1) goes to one of its neighbor $v_{i_{2}}$; or

(2) goes towards one of its neighbors $v_{i_{2}}$, but does not reach it, and goes back to $v_{i_{1}}$.

In case (1), we define the super-walk to have a negative sign, and in case (2), we define the super-walk to have a positive sign.

A \emph{super-walk} of length $k$ from $v_{i}$ to $v_{j}$ is formed by repeating (1) and/or (2) $k$ times, starting at $v_{i}$ and ending at $v_{j}$, where we determine the sign by multiplying together the signs of the steps.

\end{definition}

\begin{example}\label{eg_2}

In Figure~\ref{Figure 1}: (1) $v_{1}\overset{e_{1}}{\rightarrow}v_{2}$ is a super-walk from $v_{1}$ to $v_{2}$ of length 1; (2) $v_{1}\overset{e_{1}}{\rightarrow}v_{1}$ and $v_{1}\overset{e_{2}}{\rightarrow}v_{1}$ are super-walks from $v_{1}$ to $v_{1}$ of length 1; (3) $v_{5}\overset{e_{4}}{\rightarrow}v_{5}$, $v_{5}\overset{e_{5}}{\rightarrow}v_{5}$, and $v_{5}\overset{e_{7}}{\rightarrow}v_{5}$ are super-walks from $v_{5}$ to $v_{5}$ of length 1. Furthermore, we say (1) is a super-walk with negative sign, while (2) and (3) are super-walks with positive sign.

Moreover, (4) $v_{1}\overset{e_{1}}{\rightarrow}v_{2}\overset{e_{1}}{\rightarrow}v_{2}$, $v_{1}\overset{e_{1}}{\rightarrow}v_{2}\overset{e_{3}}{\rightarrow}v_{2}$, $v_{1}\overset{e_{1}}{\rightarrow}v_{1}\overset{e_{1}}{\rightarrow}v_{2}$, $v_{1}\overset{e_{2}}{\rightarrow}v_{1}\overset{e_{1}}{\rightarrow}v_{2}$, and $v_{1}\overset{e_{2}}{\rightarrow}v_{4}\overset{e_{3}}{\rightarrow}v_{2}$ are super-walks from $v_{1}$ to $v_{2}$ of length 2. Furthermore, $v_{1}\overset{e_{1}}{\rightarrow}v_{2}\overset{e_{1}}{\rightarrow}v_{2}$, $v_{1}\overset{e_{1}}{\rightarrow}v_{2}\overset{e_{3}}{\rightarrow}v_{2}$, $v_{1}\overset{e_{1}}{\rightarrow}v_{1}\overset{e_{1}}{\rightarrow}v_{2}$, $v_{1}\overset{e_{2}}{\rightarrow}v_{1}\overset{e_{1}}{\rightarrow}v_{2}$ have one negative sign and one positive sign, thus we define each them with a negative sign. While $v_{1}\overset{e_{2}}{\rightarrow}v_{4}\overset{e_{3}}{\rightarrow}v_{2}$ has two negative signs, thus we define it with a positive sign.

\end{example}

We now present one of our two main results, relating the number of super-walks with sign, to entries of power of the even Laplacian.\\

\begin{theorem}\label{SW}

Let $\Delta^+_\Gamma$ be the even Laplacian of a graph $\Gamma$. $(\Delta {^{+}_{\Gamma }})^{k}(i,j)=\sum_{\gamma ,i\rightarrow j,k} sgn(\gamma )$ where $\gamma ,i\rightarrow j,k$ is a super-walk $\gamma$ that starts at vertex $i$, ends at vertex $j$, and has length $k$.

\end{theorem}

\begin{proof}
Proof by induction on $k$. When $k=1$, we have: $(\Delta {^{+}_{\Gamma }})^{k}=(\Delta {^{+}_{\Gamma }})^{1}$.

Note that the even graph Laplacian can also be calculated by $$\Delta ^{+}_{\Gamma } :=\Delta ^{+}_{\Gamma }(i,j)=\left\{\begin{matrix}
\text{val}(i) & \text{if } i=j & (1)\\ 
-1 & \text{if } i\text{ is adjacent to }j & (2)\\ 
0 & \text{otherwise} & (3)
\end{matrix}\right.$$ where val$(i)$ is the \emph{valence} of a vertex $v$, which equals to the number of neighbors of $v$.

In case (1), $(\Delta ^{+}_{\Gamma })(i,i)=val(i)=$ the number of neighbors of $i=$ the number of super-walks of length 1 goes from vertex $i$ back to vertex $i$. Thus, the claim holds in case (1).

In case (2), $(\Delta ^{+}_{\Gamma })(i,j)=-1$ if $i$ is adjacent to $j$. In this case, there is only one super-walk from vertex $i$ to one of its neighbors vertex $j$, and the sign is negative. Thus, the claim holds in case (2).

In case (3), $(\Delta ^{+}_{\Gamma })(i,j)=0$ if vertex $i$ is not adjacent to vertex $j$. Since no super-walk of length 1 can go from vertex $i$ to vertex $j$, the claim holds in case (3).

Therefore, the claim holds when $k=1$.

Suppose the claim holds for some natural number $k$. Consider $k+1$.

\begin{align*}(\Delta {^{+}_{\Gamma }})^{k+1}&=(\Delta {^{+}_{\Gamma }})^{k}\cdot (\Delta {^{+}_{\Gamma }})^{1} \\ \text{So, }(\Delta {^{+}_{\Gamma }})^{k+1}(i,j)&=\sum_{q=1}^{n}(\Delta {^{+}_{\Gamma }})^{k}(i,q)(\Delta ^{+}_{\Gamma })(q,j) \\
&=(\Delta {^{+}_{\Gamma }})^{k}(i,j)\cdot (\Delta^{+} _{\Gamma })(j,j)+\overbrace{\sum_{q=1}^{n}(\Delta {^{+}_{\Gamma }})^{k}(i,q)\cdot (-1)}^{\text{where } q \text{ and } j \text{ adjacent}}+\overbrace{\sum_{q=1}^{n}(\Delta {^{+}_{\Gamma }})^{k}(i,q)\cdot 0}^{\text{where } q \text{ and } j \text{ non-adjacent}} \\ 
&=\sum_{\gamma ,i\rightarrow j,k} \text{sgn}(\gamma )\cdot \text{val}(j)+\sum_{\gamma ,i\rightarrow j,k} \text{sgn}(\gamma )\cdot (-1)+0.
\end{align*}

Let $\sum_{\gamma ,i\rightarrow j,k} \text{sgn}(\gamma )\cdot \text{val}(j)=X$. In $X$, we first count the super-walks from vertex $i$ to vertex $j$ of length $k$ by calculating $\sum_{\gamma ,i\rightarrow j,k} \text{sgn}(\gamma )$. Now we stand at vertex $j$, and have to go from vertex $j$ to vertex $j$ of length 1. Obviously, the number of the last step equals to val$(j)$. Since val$(j)$ is always non-negative, $X$ calculates $\sum_{\gamma} \text{sgn}(\gamma )$ where $\gamma$ is one of the super-walks from vertex $i$ to vertex $j$ of length $k+1$, and the last step going from vertex $j$ to vertex $j$.

Let $\sum_{\gamma ,i\rightarrow j,k} \text{sgn}(\gamma )\cdot (-1)=Y$. In $Y$, we first count the super-walks from vertex $i$ to vertex $q$ of length $k$ by calculating $\sum_{\gamma ,i\rightarrow q,k} \text{sgn}(\gamma )$. Now we stand at vertex $q$, a neighbor of vertex $j$, and have to go from vertex $q$ to vertex $j$ of length 1. There is only one way to go from vertex $q$ to vertex $j$ of length 1, and its sign is negative. That is $\sum_{\gamma ,q\rightarrow j,1} \text{sgn}(\gamma )=-1$. Hence, $Y$ calculates $\sum_{\gamma} \text{sgn}(\gamma )$ where $\gamma$ is one of the super-walks from vertex $i$ to vertex $j$ of length $k+1$, and the last step going from vertex $q$ to vertex $j$.

Therefore, $X+Y$ calculates $\sum_{\gamma ,i\rightarrow j,k+1} \text{sgn}(\gamma )$. Hence, the claim holds when $k$ is replaced by $k+1$.

This is the end of the proof.

\end{proof}

\section{Edge Super-Walks}\label{SESW}

In this section, we define edge super-walk. Similar to the super-walk, edge super-walk is a sign-sensitive edge-to-edge walk in an oriented graph. Particularly, the sign of edge super-walk depends on not only the way each step goes, but also the orientation of the graph, which is defined in Definition~\ref{signdef} in this section. Theorem~\ref{ESW} shows that the total signs of edge super-walks of fixed length can be calculated by the corresponding odd graph Laplacian and its powers.\\

\begin{definition}\label{signdef}

Given an oriented graph $\Gamma$. Let $i$ be an edge in $\Gamma$ which starts at vertice $A$ and ends at vertice $B$. We say \emph{the sign of $A$ in $i$ }is negative, and \emph{the sign of $B$ in $i$ }is positive.

\end{definition}

\begin{example}\label{eg_3}

In the figure~\ref{Figure 2}, we say that the sign of $v_{5}$ in $e_{4}$ is positive; the sign of $v_{5}$ in $e_{5}$ in negative; the sign of $v_{5}$ in $e_{7}$ is positive.

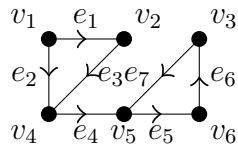
\begin{figure}[h]
\begin{center}
	\begin{tikzpicture}
	\tikzset{
	mid arrow/.style={
	decoration={markings,mark=at position 0.5 with {\arrow[scale = 2]{>}}},
	postaction={decorate},
	shorten >=0.4pt}}
	\path (0,0) coordinate (X4); \fill (X4) circle (3pt);
	\path (0,1) coordinate (X1); \fill (X1) circle (3pt);
	\path (1,1) coordinate (X2); \fill (X2) circle (3pt);
	\path (1,0) coordinate (X5); \fill (X5) circle (3pt);
	\path (2,0) coordinate (X6); \fill (X6) circle (3pt);
	\path (2,1) coordinate (X3); \fill (X3) circle (3pt);
	\node[above left] at (X1) {$v_{1}$}; 
	\node[above right] at (X2) {$v_{2}$}; 
	\node[above right] at (X3) {$v_{3}$};
	\node[below left] at (X4) {$v_{4}$};
	\node[below] at (X5) {$v_{5}$};
	\node[below right] at (X6) {$v_{6}$};
	\draw[mid arrow] (X1) -- (X2) node[midway,above]{$e_{1}$};
	\draw[mid arrow] (X1) -- (X4) node[midway,left]{$e_{2}$};
	\draw[mid arrow] (X2) -- (X4) node[midway,right]{$e_{3}$};
	\draw[mid arrow] (X4) -- (X5) node[midway,below]{$e_{4}$};
	\draw[mid arrow] (X5) -- (X6) node[midway,below]{$e_{5}$};
	\draw[mid arrow] (X6) -- (X3) node[midway,right]{$e_{6}$};
	\draw[mid arrow] (X3) -- (X5) node[midway,left]{$e_{7}$};
	\end{tikzpicture}
	\end{center}
	\caption{The graph used in Example~\ref{eg_3} and Example~\ref{eg_4}}
\label{Figure 2}\end{figure}

\end{example}

\begin{definition}

Given an oriented graph $\Gamma$ with vertices $v_{1}, v_{2}, \ldots, v_{|V|}$, and edges $e_{1}, e_{2}, \ldots, e_{|E|}$, we have:

An \emph{edge super-walk} of length 1 is a walk starts at $e_{i_{1}}$, and: 

(1) goes to one of its neighbor edge $e_{i_{2}}$ through the vertex $v_{i}$, where $e_{i_{1}}$ and $e_{i_{2}}$ intersect at $v_{i}$; or

(2) goes through one of the end points and back to itself.

In case (1), if the sign of $v_{i}$ in $e_{i_{1}}$ and the sign of $v_{i}$ in $e_{i_{2}}$ are the same, we define this walk to have a positive sign; if not, we define this walk to have a negative sign. In case (2), the walk is defined to have a positive sign.

An \emph{edge super-walk} of length $k$ from $e_{i}$ to $e_{j}$ is formed by repeating (1) and/or (2) $k$ times, starting at $e_{i}$ and ending at $e_{j}$, where we determine the sign by multiplying together the signs of the steps.

\end{definition}

\begin{example}\label{eg_4}

In Figure~\ref{Figure 2}: (1) $e_{1}\overset{v_{1}}{\rightarrow}e_{2}$ is an edge super-walk from $e_{1}$ to $e_{2}$ of length 1; (2) $e_{1}\overset{v_{1}}{\rightarrow}e_{1}$ and $e_{1}\overset{v_{2}}{\rightarrow}e_{1}$ are edge super-walks from $e_{1}$ to $e_{1}$ of length 1. Furthermore, we say (1) is an edge super-walk with a negative sign, while both walks in (2) have a positive sign.

Moreover, $e_{1}\overset{v_{1}}{\rightarrow}e_{1}\overset{v_{1}}{\rightarrow}e_{2}$, $e_{1}\overset{v_{2}}{\rightarrow}e_{1}\overset{v_{1}}{\rightarrow}e_{2}$, $e_{1}\overset{v_{1}}{\rightarrow}e_{2}\overset{v_{1}}{\rightarrow}e_{2}$, $e_{1}\overset{v_{1}}{\rightarrow}e_{2}\overset{v_{4}}{\rightarrow}e_{2}$, and $e_{1}\overset{v_{2}}{\rightarrow}e_{3}\overset{v_{4}}{\rightarrow}e_{2}$ are edge super-walks from $e_{1}$ to $e_{2}$ of length 2. Furthermore, all of the first four edge super-walks have a positive sign, while the last one has a negative sign.

\end{example}

We now present the other main result, relating the number of edge super-walks with sign, to entries of power of the odd Laplacian.\\

\begin{theorem}\label{ESW}

Let $\Delta^-_\Gamma$ be the odd Laplacian of a graph $\Gamma$. $(\Delta _{\Gamma }^{-})^{k}=\sum_{\gamma ,i\rightarrow j,k} sgn(\gamma )$ where $\gamma ,i\rightarrow j,k$ is an edge super-walk $\gamma$ that starts at edge $i$, ends at edge $j$, and has length $k$.

\end{theorem}

\begin{proof}
Prove by induction on $k$. When $k=1$, we have: $(\Delta _{\Gamma }^{-})^{k}=(\Delta _{\Gamma }^{-})^{1}=I^{t}I$.

Therefore, $\Delta _{\Gamma }^{-}(i,j)=\sum_{q=1}^{n}I_{qi}I_{qj}$.

By the definition of incidence matrix, $I_{ij}=0 \text{ or } 1 \text{ or } -1$.

Case (1), if $i=j$, we have $\Delta _{\Gamma }^{-}(i,i)=\sum_{q=1}^{n}I_{qi}I_{qi}=2$.

Meanwhile, we have two edge super-walks of length 1 going from one edge to itself, and both of them are positive. Hence, $\Delta _{\Gamma }^{-}(i,i)=2$ follows the claim. The claim holds in case (1).

Case (2), if $i\neq j$, and edge $i$ is not incident to edge $j$. It follows that if $I_{qi}$ is non-zero, then $I_{qj}$ must be zero. Hence, $\Delta _{\Gamma }^{-}(i,j)=\sum_{q=1}^{n}I_{qi}I_{qj}=0$.

Meanwhile, we cannot go from edge $i$ to edge $j$ of length 1 if two edges have no intersection. Hence $\Delta _{\Gamma }^{-}(i,j)=0$ follows the claim. The claim holds in case (2).

Case (3), if $i\neq j$, and edge $i$ intersects edge $j$ at vertex $A$. We have: $\Delta _{\Gamma }^{-}(i,j)=\sum_{q=1}^{n}I_{qi}I_{qj}=\pm 1$.

Since only non-zero items contribute to the result, if the sign of $A$ in $i$ is the same as sign of $A$ in $j$, then we get $(+1)$; if the sign of $A$ in $i$ is different from sign of $A$ in $j$, then we get $(-1)$.

Meanwhile, if the sign of $A$ in $i$ and the sign of $A$ in $j$ are the same, then both edges end at $A$ or start at $A$. We have one edge super-walk of length 1 going from edge $i$ to edge $j$, and it is positive. Hence, $\Delta _{\Gamma }^{-}(i,j)=1$ follows the claim. Also, if the sign of $A$ in $i$ and the sign of $A$ in $j$ are different, then one of the edges ends at $A$, and the other one starts at $A$. We have one edge super-walk of length 1 going from edge $i$ to edge $j$, and it is negative. Hence, $\Delta _{\Gamma }^{-}(i,j)=-1$ follows the claim. The claim holds in case (3).

Therefore, the claim holds when $k=1$.

Suppose the claim holds for some natural number $k$. Consider $k+1$.

\begin{align*}(\Delta {^{-}_{\Gamma }})^{k+1}&=(\Delta {^{-}_{\Gamma }})^{k}\cdot (\Delta {^{-}_{\Gamma }})^{1}\\
\text{So, }(\Delta {^{-}_{\Gamma }})^{k+1}(i,j)&=\sum_{q=1}^{n}(\Delta {^{-}_{\Gamma }})^{k}(i,q)(\Delta ^{-}_{\Gamma })(q,j) \\
&=(\Delta _{\Gamma }^{-})^{k}(i,j)\cdot (\Delta _{\Gamma }^{-})(j,j)+\overbrace{\sum_{q=1}^{n}(\Delta _{\Gamma }^{-})^{k}(i,q)\cdot (\Delta _{\Gamma }^{-})(q,j)}^{\text{where } q \text{ and } j \text{ adjacent}}+\overbrace{\sum_{q=1}^{n}(\Delta _{\Gamma }^{-})^{k}(q,j)\cdot 0}^{\text{where } q \text{ and } j \text{ non-adjacent}} \\ 
&=\sum_{\gamma ,i\rightarrow j,k} \text{sgn}(\gamma )\cdot (\Delta _{\Gamma }^{-}) (j,j)+\sum_{\gamma ,i\rightarrow j,k} \text{sgn}(\gamma )\cdot (\Delta _{\Gamma }^{-}) (q,j)+0.
\end{align*}

Let $\sum_{\gamma ,i\rightarrow j,k} \text{sgn}(\gamma )\cdot (\Delta _{\Gamma }^{-}) (j,j)=X$. In $X$, we first count the edge super-walks from edge $i$ to edge $j$ of length $k$ by calculating $\sum_{\gamma ,i\rightarrow j,k} \text{sgn}(\gamma )$. Now we stand at edge $j$, and have to go from edge $j$ to edge $j$ of length 1. The number of the last step is 2, because every edge has two ends. Since $(\Delta _{\Gamma }^{-}) (j,j)$ is 2 for all $j$, $X$ calculates $\sum_{\gamma} \text{sgn}(\gamma )$ where $\gamma$ is the edge super-walks from edge $i$ to edge $j$ of length $k+1$, and the last step going from edge $j$ to edge $j$.

Let $\sum_{\gamma ,i\rightarrow j,k} \text{sgn}(\gamma )\cdot (\Delta _{\Gamma }^{-}) (q,j)=Y$. In $Y$, we first count the edge super-walks from edge $i$ to edge $q$ of length $k$ by calculating $\sum_{\gamma ,i\rightarrow q,k} \text{sgn}(\gamma )$. Now we stand at edge $q$, an edge connected to the edge $j$, and have to go from edge $q$ to edge $j$ of length 1. There is only one way to go from edge $q$ to edge $j$ of length 1, and its sign can be negative or positive. That is $\sum_{\gamma ,q\rightarrow j,1} \text{sgn}(\gamma )=-1=(\Delta _{\Gamma }^{-}) (q,j)$, or $\sum_{\gamma ,q\rightarrow j,1} \text{sgn}(\gamma )=1=(\Delta _{\Gamma }^{-}) (q,j)$. Hence, $Y$ calculates $\sum_{\gamma} \text{sgn}(\gamma )$ where $\gamma$ is the edge super-walks from edge $i$ to edge $j$ of length $k+1$, and the last step going from edge $q$ to edge $j$.

Therefore, $X+Y$ calculates $\sum_{\gamma ,i\rightarrow j,k+1} \text{sgn}(\gamma )$. Hence, the claim holds when $k$ is replaced by $k+1$.

This is the end of the proof.

\end{proof}

\section{Conclusion and Perspectives}

The combinatorial interpretation of the Laplacian has a physical counterpart. In quantum mechanics, we have the evolution of quantum states $\psi \in \mathcal H $ which is calculated via the so called \emph{Schr\"odinger equation:} $\frac{\partial }{\partial t}\psi=\mathbb{H}\psi$, where $\mathbb{H}$ is the Schr\"odinger operator. One can also calculate it in another way by looking at the wave equation. $\psi(y,t_{1})=\int_{\mathbb{R}}k(t_{1}-t_{0},x,y)\psi(x,t_{0})dx$. This integral is the \textit{heat integral kernel}, where $k$ is \textit{a propagator of the quantum system}. In order to compute such heat kernel, Feynman proposed the following description: $k(t_{1}-t_{0},x,y)=\int_{\mathbb{\gamma}}\mathcal{D}(\Gamma)e^{i/\hbar}S_{cl}(\gamma)$, where the integral is over paths $\gamma$ in $\mathbb{R}$, and $\mathcal{D}(\Gamma)$ is a measure on the space of walks. The two types of super-walks and their theorems given by this paper show the discrete form of the walk expansion for the partition function of the super-symmetric version of graph quantum mechanics. In this combinatorial model of quantum mechanics, we know that $\psi_{t_{1}}=exp((t_{0}-t_{1})\Delta)\psi_{t_{0}}$, where $\Delta$ is the super-Laplacian (the direct sum of the even and odd graph Laplacians). Thus we can calculate $\psi$ by summing all $\psi_{t}$'s: $\psi = \int_{\mathbb{R}}\sum_{\gamma}\frac{t^{\left | \gamma \right |}}{\left | \gamma \right |!}e^{-kt}\psi(x,t_{0})dx$, where $\gamma$ are all the super-walks defined in this paper.\\

The formulae in this paper also suggest a natural extension of walks to CW-complexes, with a given orientation, and such expressions would lead to a combinatorial interpretation of the partition function of super-symmetric quantum mechanics for CW-complexes. In del Vecchio's work, walk sum formulae and gluing formulae are derived for the case of graph quantum mechanics , and it is conjectured that this approach is still useful in the more general setting of Laplacians for CW-complexes \cite{delVeccio}. In this case the graph Laplacian is replaced by the Hodge Laplacian $\left \langle u|(\Delta_{X})^{k}|v \right \rangle = \left \langle u|(d^{*}_{X}d_{X})^{k}|v \right \rangle +\left \langle u|(d_{X}d^{*}_{X})^{k}|v \right \rangle$, where $X$ is the CW-complex such that $\Delta_{X}:\bigoplus _{i=0}^{n}Span_{\mathbb{C}}(C_{i}(X))\rightarrow Span_{\mathbb{C}}(C_{i}(X))$. Since any graph is a 1-dimensional CW-complex, one can extend the two types of super-walks defined in this paper to CW-complexes so that $\left \langle u|(\Delta_{X})^{k}|v \right \rangle = \sum_{\gamma}\text{sgn}(\gamma)$, where $X$ is the CW-complex, $u$ and $v$ are cells, $\gamma$ are the super-walks from $u$ to $v$ of length $k$. The super-walks for CW-complexes correspond to sequences of cells with alternating parity in dimension.\\

\end{document}